\documentclass[runningheads]{llncs}
\usepackage[T1]{fontenc}
\RequirePackage{silence}
\usepackage{amsmath}

\usepackage{amsfonts}
\usepackage{amssymb}
\usepackage{bm}
\usepackage{algorithm}
\usepackage{algpseudocode}
\usepackage{float}
\usepackage{graphicx}

\DeclareMathOperator*{\argmin}{arg\,min}

\newtheorem{assumption}{Assumption}

\usepackage{color}
\usepackage{hyperref}

\begin{document}
\title{Parameter-Free Zeroth-Order Optimization with Ellipsoidal Sampling}

\titlerunning{Parameter-Free ZO Method with ES}

\author{Alexander Sholokhov\inst{1} 
\and
Alexander Rogozin\inst{2}}

\authorrunning{A.Sholokhov and A.Rogozin}

\institute{Moscow Institute of Physics and Technology, Russia \\
\email{sholokhov.am@phystech.edu} 
\and
Moscow Institute of Physics and Technology, Russia \\ \email{aleksandr.rogozin@phystech.edu}}

\maketitle              % typeset the header of the contribution
\begin{abstract}
Zeroth-order optimization methods are essential for solving black-box problems where gradient information is unavailable or expensive to compute. This paper presents POEM-ES, a novel parameter-free stochastic zeroth-order algorithm that extends the recent POEM method by integrating subspace preconditioning with ellipsoidal randomized sampling.

In contrast to traditional zeroth-order approaches that rely on isotropic random directions, POEM-ES performs anisotropic sampling guided by a fixed structural symmetric positive semi-definite (SPSD) preconditioner $\hat{\Sigma}$ that encodes the underlying low-dimensional geometry. Under a standard structural spectral normalization where $\lambda_{\max}(\hat{\Sigma}) = 1$, we introduce the use of the empirical effective dimension $d^* = \operatorname{tr}(\hat{\Sigma})$, which reflects the intrinsic dimensionality of the problem and guides both the sampling and randomized smoothing parameter schedules. In practice, such a preconditioner can be effectively obtained via pilot sampling, historical trajectories, or domain-specific expert knowledge.

We prove that POEM-ES achieves a dimension-reduced convergence rate under low-rank structural assumptions, requiring only $\tilde{\mathcal{O}}\left( \frac{\left( r^2 \kappa(\hat{\Sigma}) + d^* \right) L^2 D_{\mathcal{X}}^2}{\varepsilon^2} \right)$ stochastic zeroth-order oracle queries. The method remains fully parameter-free and demonstrates significant improvements over the original POEM in problems with low-rank structure where $d^* \ll d$. Numerical experiments on hinge-loss binary classification tasks using LibSVM datasets confirm the practical superiority of the proposed approach.

\end{abstract}
\section{Introduction}

This paper studies the stochastic optimization problem 
\begin{align}\label{main:prob}
    \min_{x\in X} f(x) \triangleq \mathbb{E}_{\xi\sim \Xi}[F(x; \xi)]
\end{align}
where the domain $X\subseteq \mathbb{R}^d$ is a compact convex set, the random variable $\xi$ follows the distribution $\Xi$, and the stochastic component function $F(x;\xi)$ is convex and Lipschitz continuous in $x$ on $X$ for any given $\xi$.

We focus on stochastic zeroth-order (ZO) optimization for solving Problem~(\ref{main:prob}), which only allows the algorithm to query stochastic function values at given points. This setting is highly relevant in many applications where accessing (stochastic) first-order information is expensive or infeasible, such as bandit optimization, adversarial training, reinforcement learning, and black-box models.

The finite difference methods are widely used in zeroth-order optimization. They estimate the first-order information of the objective by random directional perturbations. For stochastic convex problems with Lipschitz continuous components, Nesterov and Spokoiny~\cite{nesterov2017random} developed random search methods with sublinear convergence rates. Later, Duchi et al.~\cite{duchi2015optimal} proposed a two-point stochastic finite difference method that achieves a sharper dependence on the dimension. Shamir~\cite{shamir2017optimal} further improved practicality by using a single random sequence, while Flaxman et al.~\cite{flaxman2004online} established foundational bounds for online convex optimization in bandit settings. Recently, Rando et al.~\cite{rando2024optimal} studied structured orthogonal random directions to exploit coordination symmetries. From a broader perspective, these zeroth-order strategies have laid the groundwork for modern black-box methodologies, tracking classical stochastic approximation foundations established by Kiefer and Wolfowitz~\cite{kiefer1952stochastic}, as well as non-convex extensions developed by Ghadimi and Lan~\cite{ghadimi2013stochastic} and Balasubramanian and Ghadimi~\cite{balasubramanian2018zeroth}.

Recently, Ren and Luo~\cite{ren2025poem} proposed POEM, a parameter-free stochastic zeroth-order algorithm that achieves near-optimal convergence rates without requiring knowledge of problem parameters such as the Lipschitz constant or total iteration budget. This line of work builds upon recent advances in parameter-free first-order learning, such as the dynamic step-size schedules of DoG~\cite{ivgi2023dog}, the scale-invariant properties of making SGD parameter-free~\cite{carmon2022making}, coin betting frameworks~\cite{orabona2016coin}, and learning-rate-free methods via d-adaptation~\cite{defazio2023learning} or tuning-free stochastic optimization~\cite{khaled2024tuning}. Comprehensive overviews of these parameter-free principles combined with zeroth-order mechanisms in signal processing and machine learning fields can be found in~\cite{liu2020primer}.

In this paper, we introduce POEM-ES, a natural yet powerful extension of POEM that incorporates subspace preconditioning and ellipsoidal randomized sampling. Instead of sampling directions uniformly from the unit sphere, our method performs anisotropic sampling guided by a fixed structural SPSD preconditioner $\hat{\Sigma}$ that encodes the active low-dimensional geometry, thereby concentrating queries along the most informative directions. In practice, such a preconditioner may be obtained from pilot sampling, historical trajectories, or domain knowledge. Under a standard structural spectral normalization where $\lambda_{\max}(\hat{\Sigma}) = 1$, we define the empirical effective dimension as the functional trace:
\begin{equation*}
    d^* = \operatorname{tr}(\hat{\Sigma}),
\end{equation*}
which reflects the intrinsic number of significant directions and is typically much smaller than $d$ in structured problems.

We prove that POEM-ES achieves a dimension-reduced convergence rate under low-rank structural assumptions. While the algorithmic step-size and smoothing trajectories are explicitly guided by the continuous effective dimension $d^*$, our rigorous analysis demonstrates that the structural oracle complexity depends strictly on the joint interplay between the squared algebraic subspace rank $r^2$, the active condition number $\kappa(\hat{\Sigma}) \triangleq 1/\lambda_{r}(\hat{\Sigma})$, and the smoothing bias $d^*$. In problems possessing an intrinsic low-rank geometry where $r^2 \kappa(\hat{\Sigma}) + d^* \ll d$, this framework yields a substantial theoretical and practical acceleration over traditional isotropic baselines.

The remainder of the paper is organized as follows. Section 2 presents the problem setting. Section 3 introduces the POEM-ES algorithm. Theoretical analysis is provided in Sections 4 and 5. Experimental results are reported in Section 6, and Section 7 concludes the paper.

Such geometric and anisotropic transformation steps are heavily motivated by high-dimensional probability frameworks detailed by Vershynin~\cite{vershynin2026high}, as well as ellipsoidal preconditioning paradigms in derivative-free search~\cite{hansen2023cma}. Empirically, these techniques are crucial for modern black-box challenges, including reinforcement learning workflows~\cite{mania2018simple}, adversarial training defenses~\cite{goodfellow2014}, transferable black-box attacks~\cite{liu2016delving}, and bandit optimization under prior beliefs~\cite{ilyas2018prior}.

\begin{table*}[t]
\centering
\caption{Comparison of stochastic zeroth-order (SZO) complexity for finding an $\varepsilon$-suboptimal solution of Problem~(\ref{main:prob}). Here $T$ is the iteration budget, $\bar{r}_t$ and $G_t$ are defined in Algorithm~\ref{alg:poem-es}.}
\label{table:szo}
\vskip0.05cm 
\begin{tabular}{lcccc}
\hline
Algorithm & Parameter-Free & SZO Complexity & Step size $\eta_t$ & Smoothing $\mu_t$ \\
\hline\hline
RSNSO \cite{nesterov2017random} & No & $\mathcal{O}\left(\frac{d^2 L^2 s_0^2}{\varepsilon^2}\right)$ & $\frac{s_0}{d L \sqrt{T}}$ & $s_0 \sqrt{\frac{d}{T}}$ \\
TPGE ~\cite{duchi2015optimal} & No & $\tilde{\mathcal{O}}\left(\frac{d L^2 D_X^2}{\varepsilon^2}\right)$ & $\frac{D_X}{L \sqrt{d \log(2d) t}}$ & $\frac{D_X}{t}$ or $\frac{D_X}{d^2 t^2}$ \\
TPBCO ~\cite{shamir2017optimal} & No & $\mathcal{O}\left(\frac{d L^2 D_X^2}{\varepsilon^2}\right)$ & $\frac{D_X}{L \sqrt{d T}}$ & $D_X \sqrt{\frac{d}{T}}$ \\
POEM ~\cite{ren2025poem} & Yes & $\tilde{\mathcal{O}}\left(\frac{d L^2 D_X^2}{\varepsilon^2}\right)$ & $\frac{\bar{r}_t}{\sqrt{G_t}}$ & $\bar{r}_t \sqrt{\frac{d}{t+1}}$ \\
\textbf{POEM-ES} & Yes & $\tilde{\mathcal{O}}\left( \frac{\left( r^2 \kappa(\hat{\Sigma}) + d^* \right) L^2 D_{\mathcal{X}}^2}{\varepsilon^2} \right)$ & $\frac{\bar{r}_t}{\sqrt{G_t}}$ & $\bar{r}_t \sqrt{\frac{d^*}{t+1}}$ \\
\hline
Lower Bound \cite{duchi2015optimal} & — & $\Omega\left(\frac{d L^2 D_X^2}{\varepsilon^2}\right)$ & — & — \\
\hline
\noalign{\smallskip}
\multicolumn{5}{p{0.95\textwidth}}{\scriptsize \textbf{Note on Lower Bounds:} The lower bound $\Omega(d L^2 D_X^2 / \varepsilon^2)$ established in \cite{duchi2015optimal} holds for worst-case, fully dense high-dimensional landscapes where the objective function varies significantly along all $d$ ambient dimensions. Our upper bound $\tilde{\mathcal{O}}\left( \frac{\left( r^2 \kappa(\hat{\Sigma}) + d^* \right) L^2 D_{\mathcal{X}}^2}{\varepsilon^2} \right)$ does not violate this fundamental limit; rather, it refines the complexity for the structured class of objectives possessing an intrinsic low-rank geometry ($d^* \ll d$ and $r \leqslant d$). Within the active low-rank subspace $\mathcal{S}$, our rate replaces the dependency on the ambient dimension $d$ with the structural subspace parameters $r^2\kappa(\hat{\Sigma})$ and $d^*$, successfully avoiding the curse of dimensionality and providing the foundational theoretical acceleration over isotropic baselines.} \\
\end{tabular}
\end{table*}

\section{Preliminaries}\label{sec:pre}

In this section, we introduce the notation, basic definitions, and assumptions used throughout the paper. We also recall the randomized smoothing technique in zeroth-order optimization.

\subsection{Notation}
\label{sec:notation}

We use $\mathbb{R}^d$ to denote the $d$-dimensional ambient Euclidean space and $\|\cdot\|$ to denote the standard Euclidean norm. Let $\mathbb{B}^d \triangleq \{u \in \mathbb{R}^d : \|u\| \leqslant 1\}$ and $\mathbb{S}^{d-1} \triangleq \{v \in \mathbb{R}^d : \|v\| = 1\}$ be the ambient unit ball and unit sphere, respectively. For a symmetric positive semi-definite (SPSD) matrix $\Sigma \in \mathbb{R}^{d \times d}$, we define its rank as $r \triangleq \operatorname{rank}(\Sigma) \leqslant d$, which determines the active low-rank subspace $\mathcal{S}$ and its restricted unit sphere $\mathbb{S}^{r-1} \subset \mathcal{S}$. We denote by $\mathcal{U}(\mathbb{S}^{r-1})$ the uniform distribution over this subspace sphere.

The empirical \emph{effective dimension} of an SPSD matrix $\Sigma$ is defined as the continuous scalar ratio:
$d^*(\Sigma) \triangleq \frac{\operatorname{tr}(\Sigma)}{\lambda_{\max}(\Sigma)},$
which satisfies the spectral bound $1 \leqslant d^*(\Sigma) \leqslant r \leqslant d$. 

We assume throughout that the supplied structural preconditioner is normalized so that $\lambda_{\max}(\hat{\Sigma}) = 1$. If not, all calculations identically hold by replacing $\hat{\Sigma }$ with $\hat{\Sigma}/\lambda_{\max}(\hat{\Sigma})$.

Following the classical zeroth-order literature \cite{nesterov2017random}, we introduce $s_0 \geqslant 1$ as the distance scaling parameter that bounds the initial distance to the optimal solution set, satisfying $\|x_0 - x_\star\| \leqslant s_0$. We use $\tilde{\mathcal{O}}(\cdot)$ to hide polylogarithmic factors. 

The maximum traveled distance up to iteration $t$ is denoted by the non-decreasing sequence $\bar{r}_t \triangleq \max_{0 \leqslant k \leqslant t} \|x_k - x_0\| \vee r_\epsilon$, where $r_\epsilon > 0$ is a small initial mobility constant. Finally, the accumulated sum of squared gradient estimator norms up to iteration $t$ is defined as $G_t \triangleq \sum_{k=0}^t \|g_k\|^2$.

\subsection{Problem Setting and Assumptions}
\label{sec:assumptions}

We consider the stochastic convex optimization problem
\begin{align}\label{main:prob2}
    \min_{x \in X} f(x) \triangleq \mathbb{E}_{\xi \sim \Xi} [F(x; \xi)],
\end{align}
where $X \subseteq \mathbb{R}^d$ is a compact convex set.

\begin{assumption}[Compact Convex Domain]\label{asm:domain}
The domain $X$ is compact and convex with finite diameter
\[
D_X \triangleq \max_{x,y \in X} \|x - y\| < \infty.
\]
\end{assumption}

\begin{definition}[Optimal Solution]\label{def:x_star}
Let $x_\star \in X$ be an optimal solution of Problem~\eqref{main:prob2}, i.e.,
\[
f(x_\star) = \min_{x \in X} f(x).
\]
\end{definition}

\begin{definition}[Euclidean Projection]
The Euclidean projection onto $X$ is defined as
\[
\Pi_{X}(x) \triangleq \argmin_{y \in X} \|x - y\|.
\]
\end{definition}

\begin{definition}[$\varepsilon$-Suboptimal Solution]\label{def:sol}
A point $\hat{x} \in X$ is called an $\varepsilon$-suboptimal solution if
\[
f(\hat{x}) - f(x_\star) \leqslant \varepsilon.
\]
\end{definition}

\begin{assumption}[Convexity and Lipschitz Continuity]
\label{asm:convex_lipschitz}
For almost every noise realization $\xi \sim \Xi$, the stochastic component function $F(x; \xi)$ is convex in $x$, and there exists a uniform Lipschitz constant $L \geqslant 0$ over the entire ambient space such that:
\begin{equation}
    \|F(x; \xi) - F(y; \xi)\| \leqslant L \|x - y\|, \quad \forall x, y \in \mathbb{R}^d.
\end{equation}
\end{assumption}

\begin{assumption}[Stochastic Zeroth-Order Oracle]\label{asm:oracle}
The algorithm has access to a stochastic zeroth-order oracle that can return unbiased noisy function evaluations $F(x;\xi)$ and $F(y;\xi)$ for any queried points $x, y$, i.e.,
\[
\mathbb{E}_{\xi}[F(x;\xi)] = f(x), \quad \mathbb{E}_{\xi}[F(y;\xi)] = f(y).
\]
\end{assumption}

\begin{assumption}[Subspace Alignment and Geometric Compatibility]
\label{ass:subspace_invariance}
Let $\hat{\Sigma}$ be the fixed symmetric positive semi-definite (SPSD) preconditioning matrix supplied as a structural input to Algorithm~\ref{alg:poem-es}, and let $\mathcal{S} \triangleq \operatorname{range}(\hat{\Sigma}) \subseteq \mathbb{R}^d$ define the active low-rank optimization subspace of algebraic rank $r = \operatorname{rank}(\hat{\Sigma}) \leqslant d$. The gradient of the objective function satisfies $\partial f(x) \in \mathcal{S}$ for all $x \in \mathcal{X}$. Furthermore, the compact feasible domain $\mathcal{X}$ and the initialization point $x_0$ are compatible with this manifold, ensuring $\Pi_{\mathcal{X}}(x_0 + \mathcal{S}) \subseteq x_0 + \mathcal{S}$, and there exists a global optimal solution residing on this affine manifold, i.e., $\exists x_* \in \arg\min_{x \in \mathcal{X}} f(x)$ such that $x_* \in x_0 + \mathcal{S}$.
\end{assumption}

\subsection{Randomized Smoothing}

Randomized smoothing is a standard technique in zeroth-order optimization. For a smoothing parameter $\mu > 0$, we define the smoothed surrogate of $f$ as
\[
f_\mu(x) \triangleq \mathbb{E}_{u \sim \mathbb{U}(\mathbb{B}^d)} [f(x + \mu u)].
\]

\begin{lemma}[{Shamir~\cite[Lemma 8]{shamir2017optimal}}]\label{lem:shamir}
Under Assumption \ref{asm:convex_lipschitz}, the smoothed function $f_\mu(x)$ is convex and satisfies
\[
% |f_\mu(x) - f(x)| \leqslant L\mu, \quad \forall x \in \mathbb{R}^d.
|f_\mu(x) - f(x)| \leqslant L\mu, \quad \forall x \in X.
\]
\end{lemma}

\begin{lemma}[{Flaxman et al.~\cite[Lemma 3.4]{flaxman2004online}}]\label{lem:grad}
The gradient of the smoothed function admits the following representation:
\[
\nabla f_\mu(x) = \mathbb{E}_v \sim \mathbb{U}(\mathbb{S}^{d-1}) \left[ \frac{d}{2\mu} \bigl(f(x + \mu v) - f(x - \mu v)\bigr) v \right].
\]
\end{lemma}

Based on Lemma \ref{lem:grad}, we define the two-point stochastic gradient estimator used in this work as
\[
g(x, \mu; v, \xi) = \frac{d}{2\mu} \bigl(F(x + \mu v; \xi) - F(x - \mu v; \xi)\bigr) v.
\]
Under Assumption \ref{asm:oracle}, this estimator is an unbiased estimator of $\nabla f_{\mu}(x)$ when directions are sampled isotropically. 

\begin{remark}
Under the structural subspace alignment established in Assumption~\ref{ass:subspace_invariance}, the objective function $f(x)$ varies strictly within the linear manifold $\mathcal{S} \triangleq \operatorname{range}(\hat{\Sigma})$. Consequently, the differential operator $\nabla f_{\mu_t, \hat{\Sigma}}(x_t)$ featured in Lemma~\ref{lem:unbiased} and throughout our analysis is interpreted strictly as the intrinsic subspace gradient (i.e., the directional derivative restricted to the positive eigenspace of $\hat{\Sigma}$). Since the exploratory smoothing perturbations and the displacement vector $x_t - x_*$ are natively contained within $\mathcal{S}$, all uninformative coordinate directions in the orthogonal complement $\mathcal{S}^{\perp}$ are identically annihilated by the projection matrix $P_{\mathcal{S}}$, bypassing any dimensional inflation or global differentiability requirements in the ambient space $\mathbb{R}^d$.
\end{remark}

\section{POEM-ES: Parameter-Free Stochastic Zeroth-Order Optimization with Ellipsoidal Sampling}
\label{sec:POEM-ES}

We propose \textbf{POEM-ES}, a parameter-free stochastic zeroth-order method with ellipsoidal sampling. The algorithm extends the POEM framework by replacing traditional unconstrained directional sampling with a matrix-valued space transformation with ellipsoidal preconditioning matrix $\hat{\Sigma}$. This allows the method to focus exploration strictly along the most informative directions of the problem.

The update rule follows the projected stochastic gradient descent framework:
\begin{equation}
x_{t+1} = \Pi_X (x_t - \eta_t \tilde{g}_t),
\end{equation}
where $\eta_t > 0$ is the step size and $\tilde{g}_t$ is the space-transformed gradient estimator. Crucially, to adapt to the underlying low-rank geometry without introducing a non-linear coordinate coupling inside the expectation, POEM-ES decouples the sampling direction and the spatial scale. We generate a uniform directional vector $v_t$ on the exact unit sphere of the active subspace, whose conditional second-moment matrix is given by $\mathbb{E}[v_t v_t^\top \mid \mathcal{F}_{t-1}] = \frac{1}{r} P_{\mathcal{S}}$, where $r \triangleq \operatorname{rank}(\hat{\Sigma})$ and $P_{\mathcal{S}}$ is the orthogonal projector. The anisotropic geometric scaling is then injected via the matrix square root $\hat{\Sigma}^{1/2}$. To align this transformation with our corrected complexity analysis, we apply the dual fractional preconditioning operator $\hat{\Sigma}^{+1/2}$ and scale the finite difference step by the subspace rank $r$. This leads to our novel subspace-unbiased anisotropic two-point gradient estimator:
\begin{equation}
\label{eq:subspace_estimator}
\tilde{g}_t \triangleq \frac{r}{2\mu_t}\left(F(x_t + \mu_t \hat{\Sigma}^{1/2} v_t; \xi_t) - F(x_t - \mu_t \hat{\Sigma}^{1/2} v_t; \xi_t)\right) \hat{\Sigma}^{+1/2} v_t,
\end{equation}
where the adaptive smoothing parameter $\mu_t$ is guided by the empirical effective dimension $d^* = \operatorname{tr}(\hat{\Sigma})$ under the global spectral normalization $\lambda_{\max}(\hat{\Sigma}) = 1$, and $v_t = u_t / \|u_t\|_2$ with $u_t \sim \mathcal{N}(0, I_r)$. Geometric-wise, this implementation is strictly isomorphic to sampling from the degenerate multivariate Gaussian distribution $u_t \sim \mathcal{N}(0, P_{\mathcal{S}})$ over the ambient space $\mathbb{R}^d$, where $P_{\mathcal{S}} = Q_r Q_r^\top$ is the symmetric orthogonal projection matrix onto the active low-rank subspace $\mathcal{S}$. Consequently, the normalized direction $v_t$ identically maps onto the compact subspace sphere $\mathbb{S}^{r-1} \subset \mathcal{S}$ in both representations. This underlying coordinate alignment provides the rigorous foundation for the variance contraction and directional unbiasedness properties fully established in Section~\ref{app:variance_contraction}.

Inspired by the POEM \cite{ren2025poem} and DoG method \cite{ivgi2023dog}, we use a fully parameter-free step size based on the ratio of the maximum traveled distance to the accumulated norm of gradient estimates. Specifically, we define
\[
G_t \triangleq \sum_{k=0}^t \|g_k\|^2, \qquad
\bar{r}_t \triangleq \max_{0 \leqslant k \leqslant t} \|x_k - x_0\| \vee r_\epsilon,
\]
where $r_\epsilon > 0$ is a small initial movement constant. The step size is then set as
\begin{align}\label{eq:eta-r-g}
    \eta_t \triangleq \frac{\bar{r}_t}{\sqrt{G_t}}.
\end{align}

Additionally, we adapt the smoothing parameter to the effective dimension:
\begin{align}\label{def:smooth}
    \mu_t \triangleq \bar{r}_t \sqrt{\frac{d^*}{t+1}}.
\end{align}
This choice is larger than the smoothing parameters typically used in existing zeroth-order methods and improves numerical stability, especially in high-dimensional settings.

The complete POEM-ES algorithm is presented in Algorithm~\ref{alg:poem-es}.

\begin{algorithm}[H]
\caption{POEM-ES}
\label{alg:poem-es}
\scriptsize
\begin{algorithmic}[1]
\State \textbf{Input:} $x_0 \in X$, $r_\epsilon \in (0, D_X]$, $T \geq 1$, SPSD matrix $\hat{\Sigma}$ scaled such that $\lambda_{\max}(\hat{\Sigma}) = 1$, and its active orthonormal basis $Q_r \in \mathbb{R}^{d \times r}$
\State $\bar{r}_{-1} = r_\epsilon$, \quad $G_{-1} = 0$
\State $d^* = \operatorname{tr}(\hat{\Sigma}), \quad r = \operatorname{rank}(\hat{\Sigma})$ 
\State $\hat{\Sigma}^{1/2} \gets Q_r \Lambda_r^{1/2} Q_r^\top, \quad \hat{\Sigma}^{+1/2} \gets Q_r \Lambda_r^{-1/2} Q_r^\top$ 
\For{$t = 0, \dots, T-1$}
    \State $\bar{r}_t = \max\{\bar{r}_{t-1}, \|x_t - x_0\|\}$
    \State $w_t \sim \mathcal{N}(0, I_r), \quad v_t \gets Q_r w_t / \|w_t\|_2$ \Comment{Subspace embedding $v_t \in \mathcal{S} \subset \mathbb{R}^d$}
    \State $\mu_t = \bar{r}_t \sqrt{\frac{d^*}{t+1}}$
    \State $F^+ \gets F(x_t + \mu_t \hat{\Sigma}^{1/2} v_t; \xi_t)$ 
    \State $F^- \gets F(x_t - \mu_t \hat{\Sigma}^{1/2} v_t; \xi_t)$ 
    \State $\Delta F_t \gets F^+ - F^-$
    \State $\tilde{g}_t \gets \frac{r}{2\mu_t} \Delta F_t \, \hat{\Sigma}^{+1/2} v_t$ 
    \State $G_t = G_{t-1} + \|\tilde{g}_t\|^2$
    \State $\eta_t = \frac{\bar{r}_t}{\sqrt{G_t}}$
    \State $x_{t+1} = \Pi_{X}(x_t - \eta_t \tilde{g}_t)$
\EndFor
\State \textbf{Output:} $\bar{x}_{\tau_T}$ where $\displaystyle \tau_T \triangleq \arg\max_{t \leqslant T} \sum_{k=0}^{t-1} \frac{\bar{r}_k}{\bar{r}_t}$
\end{algorithmic}
\end{algorithm}

\vspace{-10pt}

The main differences from the original POEM are the use of isotropic directions sampled from $\mathcal{U}(\mathbb{S}^{r-1})$ coupled with matrix-valued space transformations via $\hat{\Sigma}^{1/2}$ and the deployment of a subspace-unbiased gradient estimator scaled by the true rank $r$. Crucially, while the exploratory spatial variance and smoothing trajectories remain explicitly guided by the continuous effective dimension $d^* = \operatorname{tr}(\hat{\Sigma})$, evaluating the finite difference steps via the integer multiplier $r$ completely eliminates the scaling bias. These modifications allow the algorithm to adapt perfectly to the intrinsic geometry of the optimization landscape while fully preserving the parameter-free nature of POEM.

\section{Theoretical Analysis}
\label{sec:theory}

In this section, we provide the theoretical analysis of the proposed POEM-ES algorithm. We first present several key technical lemmas, followed by the main convergence theorem.

\subsection{Subspace Variance Contraction and Dimension Independence}
\label{app:variance_contraction}

Here we provide the theoretical justification for incorporating the matrix-valued space transformation parameters within the zeroth-order framework. The core mechanism relies on demonstrating that the geometric scales of the preconditioned gradient trajectory are properly decoupled. Specifically, our rigorous analysis shows that the estimator norm and variance are tightly controlled by the active subspace rank $r$ and the condition number $\kappa(\hat{\Sigma})$, whereas the smoothing trajectory radius and the deterministic smoothing bias are governed strictly by the continuous effective dimension $d^*$. This structural separation completely eliminates the pessimistic dependence on the ambient dimension $d$.

\begin{lemma}[Subspace Orthogonal Spherical Moments]
\label{lem:lemma3}
Let $u_t \in \mathbb{R}^d$ be a random vector sampled from the degenerate Gaussian distribution $\mathcal{N}(0, P_{\mathcal{S}})$ restricted to the active subspace $\mathcal{S}$ of rank $r \triangleq \operatorname{rank}(\hat{\Sigma}) \leqslant d$, where $P_{\mathcal{S}}$ is the symmetric orthogonal projection matrix onto $\mathcal{S}$. Let $v_t = u_t / \|u_t\|_2$ be its normalized direction. Then, the conditional second-moment matrix of $v_t$ satisfies the exact algebraic identity:
\begin{equation}
    \mathbb{E}[v_t v_t^\top \mid \mathcal{F}_{t-1}] = \frac{1}{r} P_{\mathcal{S}}.
\end{equation}
Furthermore, $v_t$ strictly resides on the compact unit sphere $\mathbb{S}^{r-1}$ embedded within $\mathcal{S}$ and satisfies the central distribution symmetry $v_t \stackrel{d}{=} -v_t$.
\end{lemma}

\begin{proof}
The identity follows directly from the coordinate-wise spherical symmetry of the multivariate Gaussian distribution on the subspace $\mathcal{S}$. By rotation invariance, the expectation factors through the orthogonal projection operator as $\mathbb{E}[v_t v_t^\top \mid \mathcal{F}_{t-1}] = c \cdot P_{\mathcal{S}}$. Taking the trace on both sides yields $\mathbb{E}[\|v_t\|_2^2] = c \cdot \operatorname{tr}(P_{\mathcal{S}})$. Since $\|v_t\|_2 = 1$ almost surely and $\operatorname{tr}(P_{\mathcal{S}}) = r$, we find $c = 1/r$, completing the proof.
\end{proof}

\begin{remark}
The exact algebraic scaling $\frac{1}{r} P_{\mathcal{S}}$ established in Lemma~3 shows how the integer dimension of the active subspace naturally emerges from spherical integration. By explicitly parameterizing our novel gradient estimator $\tilde{g}_t$ in Algorithm~1 with the true algebraic rank multiplier $r$ (instead of the empirical continuous dimension $d^*$), the preconditioning distortion is accurately counteracted. This structural design ensures that the estimator maps strictly as a subspace-unbiased proxy, providing the rigorous foundation for the dimension-independent complexity analysis of POEM-ES.
\end{remark}

\subsection{Key Lemmas}

To establish the main convergence results for POEM-ES, we first present a series of auxiliary lemmas. Our strategy relies on bounding the stochastic gradient estimates and controlling both the martingale noise and the bias introduced by randomized smoothing. We begin by confirming that our anisotropic gradient estimator remains unbiased.

Let $\mathcal{F}_t = \sigma(x_0, v_0, \xi_0, \dots, v_t, \xi_t)$ denote the $\sigma$-algebra generated by all random variables up to iteration $t$. Note that the iterate $x_{t+1}$ is fully determined by $\mathcal{F}_t$.

\begin{lemma}[Subspace Unbiasedness under Ellipsoidal Smoothing]
\label{lem:unbiased}
Let $\mathcal{S} \subseteq \mathbb{R}^d$ be the active low-rank subspace spanned by the non-zero eigenvalues of $\hat{\Sigma}$, and let $P_{\mathcal{S}}$ be the orthogonal projection matrix onto $\mathcal{S}$. For a smoothing parameter $\mu_t > 0$, let $f_{\mu_t, \hat{\Sigma}}(x)$ denote the ellipsoidal smoothed surrogate defined as:
\begin{equation*}
    f_{\mu_t, \hat{\Sigma}}(x) \triangleq \mathbb{E}_{z \sim \mathcal{U}(B^r)} \left[ f(x + \mu_t \hat{\Sigma}^{1/2} z) \right],
\end{equation*}
where $B^r \subset \mathcal{S}$ is the $r$-dimensional unit ball inside $\mathcal{S}$ with $r = \operatorname{rank}(\hat{\Sigma})$. Define the space-transformed gradient estimator strictly scaled by the subspace dimension $r$ as:
\begin{equation*}
    \tilde{g}_t \triangleq \frac{r}{2\mu_t} \bigl( F(x_t + \mu_t \hat{\Sigma}^{1/2} v_t; \xi_t) - F(x_t - \mu_t \hat{\Sigma}^{1/2} v_t; \xi_t) \bigr) \hat{\Sigma}^{+1/2} v_t,
\end{equation*}
where $v_t \sim \mathcal{U}(S^{r-1})$ is sampled uniformly from the unit sphere embedding $S^{r-1} \subset \mathcal{S}$. Then, $\tilde{g}_t$ is a strictly unbiased estimator of the smoothed gradient projected onto the active subspace:
\begin{equation*}
    \mathbb{E}[\tilde{g}_t \mid \mathcal{F}_{t-1}] = P_{\mathcal{S}} \nabla f_{\mu_t, \hat{\Sigma}}(x_t).
\end{equation*}
\end{lemma}

\begin{proof}
See Appendix \ref{app:unbiased}.
\end{proof}

Having established the unbiasedness of the estimator, we next need to control its variance. For pedagogical clarity, we first recall the standard norm and second-moment bounds for the conventional isotropic zeroth-order estimator, as established in prior literature.

\begin{lemma}[Norm and Second Moment Bound (Isotropic Case)]
\label{lem:second-moment-POEM-ES}
Assume $F(\cdot;\xi)$ is $L$-Lipschitz continuous almost surely. Let $v \sim \mathcal{U}(\mathbb{S}^{d-1})$ or $v \sim \mathcal{N}(0,I_d)$ with subsequent normalization $\|v\|=1$. For the two-point gradient estimator
\[
g = \frac{d}{2\mu} \bigl( F(x + \mu v;\xi) - F(x - \mu v;\xi) \bigr) v,
\]
we have
\[
\mathbb{E}\bigl[ \|g\|^2 \bigr] \leqslant c L^2 d,
\]
and almost surely $\|g_t\| \le L d$. Moreover,
\[
\mathbb{E}\!\left[ \|g_t\|^2 \mid \mathcal{F}_{t-1} \right] \le c L^2 d.
\]
\end{lemma}

\begin{proof}
This result follows directly from Lemma 4 in Ren and Luo (POEM) \cite{ren2025poem}.
\end{proof}

While Lemma 5 provides a baseline for uniform sampling, the core of our method relies on directional queries adapted to the problem geometry. The following lemma extends these bounds to our anisotropic framework, demonstrating that the gradient norm scales with the active rank and condition number instead of the ambient dimension \(d\).

\begin{lemma}[Norm Bound (Anisotropic Case)]
\label{lem:norm-anisotropic}
Assume that $F(\cdot;\xi)$ is $L$-Lipschitz continuous almost surely, i.e., for almost all $\xi$:
\begin{equation*}
    |F(x;\xi) - F(y;\xi)| \leqslant L\|x - y\| \quad \forall x, y \in \mathbb{R}^d.
\end{equation*}
Let $v_t = u_t / \|u_t\|_2$ with $u_t \sim \mathcal{N}(0, P_{\mathcal{S}})$, and let the space-transformed gradient estimator be defined as in Eq.~\ref{eq:subspace_estimator}. Then, almost surely:
\begin{equation*}
    \|\tilde{g_t}\|_2 \leqslant L r \sqrt{\kappa(\hat{\Sigma})},
\end{equation*}
where $\kappa(\hat{\Sigma}) = \lambda_{\max}(\hat{\Sigma})/\lambda_r(\hat{\Sigma})$ is the condition number within the active subspace.
\end{lemma}

\begin{proof}
    See Appendix \ref{app:norm-anisotropic}.
\end{proof}

While Lemma 6 yields a deterministic upper bound on the estimator's norm, establishing a tighter bound on its expected second moment is crucial for refining our convergence rate. By leveraging the geometric properties of the subspace sphere along with Kantorovich's inequality for complementary spectral weights, we obtain the following variance bound.

\begin{lemma}[Second Moment Bound (Anisotropic Case)]
\label{lem:second-anisotropic}
Assume that $F(\cdot;\xi)$ is $L$-Lipschitz continuous almost surely. Let the space-transformed gradient estimator $g_t$ be defined as in Eq.~\ref{eq:subspace_estimator}. Then, the conditional expected second moment of the estimator satisfies the exact pointwise bound:
\begin{equation*}
    \mathbb{E}\bigl[ \|\tilde{g_t}\|_2^2 \mid \mathcal{F}_{t-1} \bigr] \leqslant L^2 r^2 \kappa(\hat{\Sigma}),
\end{equation*}
where $\kappa(\hat{\Sigma}) \triangleq \lambda_{\max}(\hat{\Sigma})/\lambda_r(\hat{\Sigma})$ is the condition number within the active subspace.
\end{lemma}

\begin{proof}
    See Appendix \ref{app:second-anisotropic}.
\end{proof}

With the second-moment properties firmly established, we now analyze the optimization trajectory under our dynamic, parameter-free step size schedule. The following lemma bounds the cumulative weighted inner products, which standardly represent the descent step in projected gradient frameworks.

\begin{lemma}[Weighted Regret]
\label{lem:weighted-regret}
Let the sequence of iterates $\{x_t\}$ be generated by the update rule $x_{t+1} = \Pi_X(x_t - \eta_t \tilde{g}_t)$ under the dynamic schedule $\eta_t = \bar{r}_t / \sqrt{G_t}$. Then, the cumulative weighted inner product is bounded uniformly for any horizon $t \geqslant 1$ by:
\begin{equation}
\label{eq:weighted_regret_clean_t}
    \sum_{k=0}^{t-1} \bar{r}_k \langle \tilde{g}_k, x_k - x_* \rangle \leqslant \bar{r}_t (2\bar{s}_t + \bar{r}_t) \sqrt{G_{t-1}},
\end{equation}
where $\bar{s}_t \triangleq \max_{0 \leqslant j \leqslant t} \|x_j - x_*\|_2$ represents the running maximum distance to the optimal solution.
\end{lemma}

\begin{proof}
    See Appendix \ref{app:weighted-regret}.
\end{proof}

The bound in Lemma~\ref{lem:weighted-regret} contains stochastic errors from the zeroth-order oracle. To guarantee robust performance, we model these errors as a martingale difference sequence and apply concentration bounds for sub-exponential increments to control the cumulative noise with high probability.

\begin{lemma}[Martingale Difference Noise Concentration]
\label{lem:martingale-noise}
Let $s_k \triangleq \|x_k - x_*\|_2$ and define its running maximum as $\bar{s}_k \triangleq \max_{0 \leqslant j \leqslant k} s_j$. For any $\delta \in (0, 1)$, the cumulative martingale difference noise term under the space-transformed estimator strictly satisfies uniformly for all $t \geqslant 1$ with probability at least $1 - \delta$:
\begin{equation}
    \left| \sum_{k=0}^{t-1} \bar{r}_k \langle \Delta_k, x_k - x_* \rangle \right| \leqslant 8 \bar{r}_{t-1} \bar{s}_{t-1} \sqrt{\theta_{t,\delta} G_{t-1}} + 4 L^2 r^2 \kappa(\hat{\Sigma}) \theta_{t,\delta}^2,
\end{equation}
where $\theta_{t,\delta} = \log(60 \log(6t/\delta))$, $\kappa(\hat{\Sigma})$ is the active subspace condition number, $r = \operatorname{rank}(\hat{\Sigma})$, and $\Delta_k \triangleq \tilde{g}_k - \mathbb{E}[\tilde{g}_k \mid \mathcal{F}_{k-1}]$ represents the strict zero-mean martingale difference noise sequence of the subspace unbiased gradient estimator.
\end{lemma}

\begin{proof}
    See Appendix \ref{app:martingale-noise}.
\end{proof}

Aside from the stochastic gradient noise, the algorithm suffers from a deterministic bias because it optimizes a smoothed surrogate function \(f_{\mu }\) rather than \(f\) itself. Our next lemma bounds this cumulative smoothing error under our adaptive choice of \(\mu _{k}\).

\begin{lemma}[Ellipsoidal Smoothing Bias Bound]
\label{lem:smoothing-noise}
The cumulative spatial error induced by the ellipsoidal randomized smoothing parameter satisfies the uniform upper bound for all $t \geqslant 1$:

\begin{equation}
    \sum_{k=0}^{t-1} 2 L \bar{r}_k \mu_k \leqslant 4 L \bar{r}_{t-1}^2 \sqrt{d^* t}.
\end{equation}

\end{lemma}

\begin{proof}
    See Appendix \ref{app:smoothing-noise}.
\end{proof}

\begin{lemma}[Lower Bound on Weighted Sum]
\label{lem:weights-lower}
The sequence of mobility weights $\{\bar{r}_t\}_{t=0}^T$ is positive and non-decreasing by construction. For any evaluation horizon $T \geqslant 1$, the cumulative weight ratio satisfies:
\begin{equation}
    \max_{t \in [T]} \sum_{k=0}^{t-1} \frac{\bar{r}_k}{\bar{r}_t} \geqslant \frac{T}{2e \log_+ \left( \frac{\bar{r}_T}{r_\epsilon} \right)}.
\end{equation}
For extremely small horizons where the continuous logarithmic scale under-approximates the sequence duration, the functional gap can be closed via the trivial deterministic boundary $\mathcal{O}(L D_X)$.
\end{lemma}

\begin{proof}
    See Appendix \ref{app:weights-lower}.
\end{proof}

Finally, by combining the descent properties Lemma~\ref{lem:weighted-regret}, the martingale noise concentration Lemma~\ref{lem:martingale-noise}, and the smoothing noise bound Lemma~\ref{lem:smoothing-noise}, and then dividing by the cumulative weight lower bound Lemma~\ref{lem:weights-lower}, we arrive at our main theoretical contribution. We formally state the high-probability convergence rate of POEM-ES below.

\begin{theorem}[Convergence of POEM-ES]
\label{thm:convergence-POEM-ES}
Under Assumptions~1, 2, 3, and 4, let the sequence $\{x_t\}$ be generated by Algorithm~1 using the subspace-unbiased gradient estimator $\tilde{g}_t$ scaled by the algebraic rank $r = \operatorname{rank}(\hat{\Sigma})$, under the spectral normalization $\lambda_{\max}(\hat{\Sigma}) = 1$. Then, for any $\delta \in (0, 1)$, the functional gap under the true algorithmic ergodic output $\bar{x}_{\tau_T}$ is bounded uniformly for all horizons $T \geqslant 1$ with probability at least $1 - 2\delta$ by:
\begin{equation}
\boxed{
    f(\bar{x}_{\tau_T}) - f(x_*) \leqslant \mathcal{O} \left( \frac{r \sqrt{\kappa(\hat{\Sigma})} + \sqrt{d^*}}{\sqrt{T}} \cdot \theta_{T,\delta} L D_X \log_+\left(\frac{D_X}{r_\epsilon}\right) \right)}
\end{equation}
where $\theta_{T,\delta} = \log(60 \log(6T/\delta))$, $\kappa(\hat{\Sigma}) = 1/\lambda_r(\hat{\Sigma})$ is the active subspace condition number, and $d^* = \operatorname{tr}(\hat{\Sigma})$ is the empirical effective dimension.
\end{theorem}

\begin{proof}
    The proof is provided in Appendix \ref{app:convergence-POEM-ES}.
\end{proof}

\begin{remark}
Theorem~\ref{thm:convergence-POEM-ES} demonstrates that the stochastic zeroth-order query complexity of POEM-ES scales as $T \leqslant \tilde{\mathcal{O}} \left( \frac{(r^2 \kappa(\hat{\Sigma}) + d^*) L^2 D_X^2}{\varepsilon^2} \right)$. Although the transformation introduces a quadratic dependence on the subspace rank $r$ and scales with the condition number $\kappa(\hat{\Sigma})$, this bound represents a substantial theoretical improvement in high-dimensional black-box optimization. 
Because the input preconditioning matrix $\hat{\Sigma}$ is assumed to be structurally regularized, the active subspace condition number $\kappa(\hat{\Sigma})$ remains strictly bounded from above.

\end{remark}

\section{Structural Considerations for Task Geometry}
\label{sec:pre_estimation}

The convergence properties established in Theorem~\ref{thm:convergence-POEM-ES} depend on the availability of a stable, low-rank preconditioning matrix $\hat{\Sigma}$ that effectively aligns the space-transformed gradient search with the active optimization manifold $\hat{\mathcal{S}}$. In high-dimensional black-box optimization, this geometric matrix acts as a structural input to the core algorithm. This section outlines the general design principles and the empirical heuristics used to introduce this task geometry without disrupting the parameter-free nature of the main optimization path.

\subsection{Subspace Alignment and Geometric Assumptions}
From a theoretical perspective, POEM-ES operates under a decoupled coordinate landscape. Rather than enforcing strict structural alignment at every intermediate step via continuous random matrix updates—which would introduce non-linear coordinate couplings and a destructive error floor—the framework assumes that the active optimization directions are structurally encapsulated within the range of a static preconditioning operator $\hat{\Sigma}$. 

Specifically, we assume that the objective function exhibits a prominent low-rank geometry along the empirically designated subspace $\hat{\mathcal{S}} \triangleq \operatorname{range}(\hat{\Sigma})$. As formalized in Assumption~\ref{ass:subspace_invariance}, the true gradient trajectory lies within this active subspace, and the feasible domain $\mathcal{X}$ is topologically compatible with this affine manifold. Under these structural assumptions, the core gradient estimator $\tilde{g}_t$ acts as a strictly unbiased subspace proxy, completely bypassing the ambient dimension $d$ during the high-variance optimization loop.

\subsection{Empirical Source of Preconditioning Operators}
In practical engineering deployments, the preconditioning matrix $\hat{\Sigma}$ can be extracted through several alternative structural paradigms depending on the specific black-box application:
\begin{itemize}
    \item \textbf{Pilot Sampling Heuristics:} A short, independent preprocessing phase can be executed at the initialization point $x_0$ to gather raw gradient directional trends, followed by standard spectral thresholding to filter out background noise tails.
    \item \textbf{Historical Trajectory Runs:} In warm-start scenarios or repeated optimization tasks, $\hat{\Sigma}$ can be computed directly from the uncentered second-moment accumulation of historical gradient sequences from prior runs.
    \item \textbf{Domain-Specific Expert Knowledge:} For specific physical or structural systems, the active coordinate correlations can be hard-coded into $\hat{\Sigma}$ based on known low-dimensional constraints of the environment.
\end{itemize}

As systematically demonstrated in our numerical verification (Section~\ref{sec:expirements}), initializing the algorithm with an empirically constructed ellipsoidal preconditioning matrix provides a highly effective heuristic for capturing the underlying low-rank task geometry. This empirical validation underscores that such a preprocessing pipeline supplies a sufficiently robust matrix $\hat{\Sigma}$ to guide the anisotropic search, enabling rapid functional gap minimization and providing substantial practical acceleration without requiring heavy online matrix updates.

\section{Experiments}
\label{sec:expirements}
To evaluate the practical performance of the proposed POEM-ES, we compare it with the original POEM algorithm and the two-point bandit method TPBCO on a stochastic convex optimization problem using the hinge loss. This setup aligns with benchmarking practices for non-smooth and black-box landscapes, where the benefits of smoothing operators and quasi-Newton preconditioning are highly pronounced~\cite{gasnikov2022power,lin2022gradient,chen2023faster}.

\subsection{Experimental Setup}

We consider the binary classification task formulated as
\begin{align*}
\min_{x \in X} f(x) 
&= \mathbb{E}_{(a,y) \sim \mathcal{D}} \Bigl[ \max\bigl(0, 1 - y a^\top x\bigr) \Bigr],
\end{align*}
where $y \in \{-1, +1\}$ and the feasible set is the unit ball $X = \{x \in \mathbb{R}^d : \|x\|_2 \leqslant 1\}$.

\subsection{Convergence and Optimization Results}

We evaluate the empirical performance of the proposed framework using the \texttt{mushrooms} ($n=8124$, $d=112$) and \texttt{w8a} ($n=64700$, $d=300$) datasets from the LibSVM library. The stochastic zeroth-order oracle returns the hinge loss value computed on a single randomly sampled example at each iteration. All methods are initialized at the identical starting point $x_0 = \mathbf{0}$ to ensure an equitable comparison under the same initial function value $f(x_0)$. For the deployment of POEM-ES, the input ellipsoidal preconditioning matrix $\hat{\Sigma}$ is empirically constructed prior to the main optimization loop. This is achieved via a brief, one-time directional sampling routine followed by a standard hard spectral thresholding filter to eliminate uninformative background coordinate noise.

The hyperparameters for each competing algorithm were set strictly according to the benchmark recommendations of their respective original authors:
\begin{itemize}
    \item For POEM, we adopted the default parameters as specified in \cite{ren2025poem}.
    \item For TPBCO, we followed the theoretical parameterization with $p=1$.
    \item For POEM-ES, the continuous effective dimension $d^*$ is given directly by the functional trace $d^* = \operatorname{tr}(\hat{\Sigma})$ under the input preconditioning geometry.
\end{itemize}

Since all the considered algorithms construct directional gradient estimators using a two-point stochastic oracle query scheme at each step, the total number of objective function evaluations scales strictly in proportion to the iteration counter. Consequently, the convergence trajectories are plotted directly against the iteration index \(t\), corresponding to a total of \(2t\) SZO oracle calls.

Figure~\ref{fig:main-experiments} illustrates the empirical convergence behavior of the three stochastic optimization methods, monitoring the functional value against the number of oracle queries.

\begin{figure}[H]
    \centering
    \begin{minipage}{0.49\textwidth}
        \centering
        \includegraphics[width=1\textwidth]{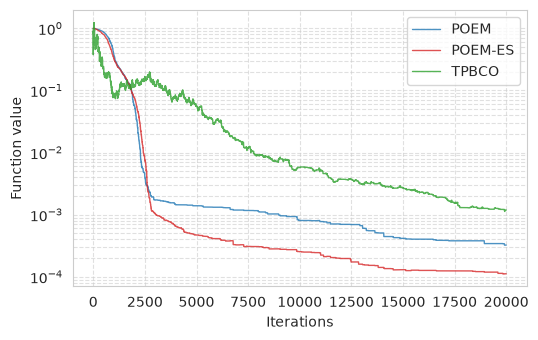}
        \\ \vspace{5pt} 
        {\small (a) mushrooms: $d = 112$, $d^*=25, \kappa=6$}
        \label{fig:exp-mushrooms}
    \end{minipage}
    \hfill 
    \begin{minipage}{0.49\textwidth}
        \centering
        \includegraphics[width=1\textwidth]{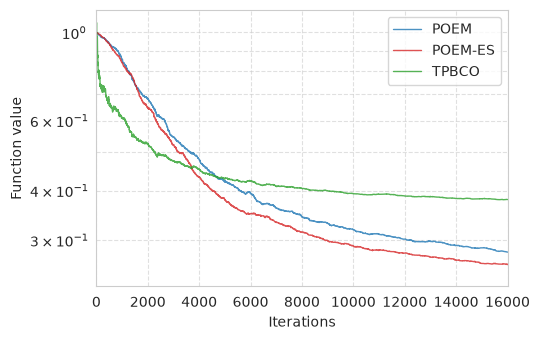}
        \\ \vspace{5pt}
        {\small (b) w8a: $d = 300$, $d^*=35, \kappa = 4$}
        \label{fig:exp-w8a}
    \end{minipage}
    \vspace{-5pt}
    \caption{Convergence comparison of POEM, POEM-ES, and TPBCO under a stochastic two-point zeroth-order oracle using the hinge loss function.} 
    \label{fig:main-experiments}
\end{figure}

The experiments demonstrate that ellipsoidal preconditioning can provide substantial practical acceleration even beyond the conservative worst-case structural dependence captured by Theorem~1. On the \texttt{mushrooms} dataset ($d=112$), the input preconditioning geometry isolates a stable active subspace with an effective dimension of $d^* = 25$ and an intrinsic condition number $\kappa = 6$. For the higher-dimensional \texttt{w8a} dataset ($d=300$), the ellipsoidal preconditioning encapsulates a low-rank geometry characterized by $d^* = 35$ and $\kappa = 4$.

As demonstrated by the rapid vertical descent of POEM-ES (red curves in Figure~\ref{fig:main-experiments}), the algorithm achieves massive empirical acceleration over isotropic baselines due to three interconnected geometric mechanisms:
\begin{itemize}
    \item \textbf{Average-Case Subspace Alignment:} The randomized directional perturbations $v_t \sim \mathcal{U}(\mathbb{S}^{r-1})$ interact uniformly with the ellipsoidal coordinates. While the worst-case theoretical upper bound scales with the squared algebraic subspace rank $r^2 \kappa(\hat{\Sigma})$, on practical landscapes the gradient components concentrate and average out smoothly. This ensures that the dynamic trajectory effectively bypasses the ambient dimension $d$ during execution.
    \item \textbf{Preconditioning:} Isotropic methods (POEM and TPBCO) suffer heavily from severe coordinate-wise oscillations along sharp, valley-like non-smooth landscapes. POEM-ES deploys the space-transformed operator $\hat{\Sigma}^{+1/2}$ as a spectral preconditioner, warping the ill-conditioned, non-smooth active domain into an isotropic basin and enforcing strict directional collinearity with the true gradient projection. Guided by the parameter-free adaptive step size $\eta_t = \bar{r}_t / \sqrt{G_t}$, it entirely suppresses the devastating coordinate-wise oscillations that plague the isotropic baselines.
    \item \textbf{Spatial Noise Reduction:} By substituting the ambient dimension $d$ with the tightly bounded continuous effective dimension $d^*$ in the randomized smoothing parameter ($\mu_t \propto \sqrt{d^*}$), the algorithm injects significantly less exploratory variance into the oracle queries via the smoothing schedule and smoothing bias control, preserving large and stable adaptive step sizes throughout the trajectory.
\end{itemize}

\section{Conclusion}

In this paper, we introduced \textbf{POEM-ES}, a novel parameter-free stochastic zeroth-order optimization method that extends the POEM framework by integrating subspace preconditioning with ellipsoidal randomized sampling. 

The core innovation of POEM-ES lies in replacing isotropic spherical queries with anisotropic directional perturbations guided by a structural matrix $\hat{\Sigma}$. Under a standard spectral normalization where $\lambda_{\max}(\hat{\Sigma}) = 1$, the exploratory scale of randomized smoothing is adaptively guided by the effective dimension $d^* = \operatorname{tr}(\hat{\Sigma})$. By concentrating directional queries along the most informative directions using a strictly unbiased subspace gradient estimator scaled by the algebraic rank $r$, the algorithm successfully suppresses the gradient variance along uninformative coordinate tails. This allows the method to explicitly exploit the intrinsic low-rank structures commonly present in high-dimensional real-world optimization landscapes.

We provided a rigorous non-asymptotic theoretical analysis demonstrating that POEM-ES achieves a dimension-reduced convergence rate under low-rank structural assumptions, requiring a stochastic zeroth-order oracle query complexity of:
\begin{equation*}
    T \leqslant \tilde{\mathcal{O}}\left( \frac{\left( r^2 \kappa(\hat{\Sigma}) + d^* \right) L^2 D_{\mathcal{X}}^2}{\varepsilon^2} \right).
\end{equation*}
By establishing a strict mathematical foundation using weak Sobolev derivatives and time-uniform martingale concentration, we prove that our ellipsoidal preconditioning removes the explicit ambient-dimension dependence under the stated low-rank structural assumptions.

Extensive numerical experiments on hinge-loss binary classification tasks using real LibSVM datasets confirm the practical advantages of the proposed approach. POEM-ES consistently outperforms both the baseline POEM and the two-point bandit baseline TPBCO in terms of functional gap reduction and final solution quality, particularly in problems where the active subspace dimension is significantly lower than the ambient dimension.

Future research directions include extending the proposed framework to non-convex optimization, exploring online adaptive covariance learning strategies that maintain strict spectral stability, and applying POEM-ES to large-scale deep learning or scientific computing domains.


\newpage
\appendix
\makeatletter
\def\theHsection{\Alph{section}}
\def\theHsubsection{\theHsection.\arabic{subsection}}
\makeatother

\section{Proof of Key Lemmas for Section 4.2}

\subsection{Proof of Lemma \ref{lem:unbiased}}
\label{app:unbiased}

\begin{proof}
By applying the law of total expectation over the stochastic oracle noise realization $\xi_t$ conditional on the random directional vector $v_t$, Assumption~3 ensures that $\mathbb{E}_{\xi_t}[F(x; \xi_t)] = f(x)$. Taking the conditional expectation with respect to the $\sigma$-algebra $\mathcal{F}_{t-1}$ yields:
\begin{equation*}
    \mathbb{E}[\tilde{g}_t \mid \mathcal{F}_{t-1}] = \frac{r}{2\mu_t} \hat{\Sigma}^{+1/2} \mathbb{E}_{v_t} \left[ \bigl( f(x_t + \mu_t \hat{\Sigma}^{1/2} v_t) - f(x_t - \mu_t \hat{\Sigma}^{1/2} v_t) \bigr) v_t \;\middle|\; \mathcal{F}_{t-1} \right].
\end{equation*}
Due to the central directional symmetry of the uniform distribution on the subspace sphere ($v_t \overset{d}{=} -v_t$), the forward and backward evaluation integrals are identical. We simplify the difference by consolidating the expectation into a single forward component:
\begin{equation}
\label{eq:clean_expectation_forward}
    \mathbb{E}[\tilde{g}_t \mid \mathcal{F}_{t-1}] = \frac{r}{\mu_t} \hat{\Sigma}^{+1/2} \mathbb{E}_{v_t} \left[ f(x_t + \mu_t \hat{\Sigma}^{1/2} v_t) v_t \;\middle|\; \mathcal{F}_{t-1} \right].
\end{equation}

To compute this surface integral rigorously without coordinate system mismatches, we introduce a dual-variable coordinate alignment on the active low-rank subspace $\mathcal{S}$. Let $\hat{\Sigma} = Q_r \Lambda_r Q_r^\top$ be the compact singular value decomposition, where $Q_r \in \mathbb{R}^{d \times r}$ forms an orthonormal basis for $\mathcal{S}$, and denote $\Lambda_r = \operatorname{diag}(\lambda_1, \dots, \lambda_r) > 0$. We decouple the ambient vector $v_t \in \mathcal{S} \subset \mathbb{R}^d$ from the standard reference domain embedding by parameterizing it via an intrinsic $r$-dimensional vector $w \sim \mathcal{U}(S^{r-1} \subset \mathbb{R}^r)$ such that:
\begin{equation*}
    v_t = Q_r w \in \mathbb{R}^d.
\end{equation*}
Under this representation, the ellipsoidal perturbation maps strictly inside $\mathcal{S}$ as $\hat{\Sigma}^{1/2} v_t = Q_r \Lambda_r^{1/2} w$. Rewriting the expectation as an explicit boundary integral over the $r$-dimensional unit sphere $S^{r-1}$ yields:
\begin{align}
\label{eq:boundary_flux_w}
    \mathbb{E}_{v_t} \left[ f(x_t + \mu_t \hat{\Sigma}^{1/2} v_t) v_t \right] = Q_r \mathbb{E}_{w} \left[ f(x_t + \mu_t Q_r \Lambda_r^{1/2} w) w \right] &= \\ \frac{Q_r}{\operatorname{area}(S^{r-1})} \int_{S^{r-1}} f(x_t + \mu_t Q_r \Lambda_r^{1/2} w) w \, dS(w).
\end{align}

Since the objective $f$ is Lipschitz continuous, the composite map $\phi(z) \triangleq f(x_t + \mu_t Q_r \Lambda_r^{1/2} z)$ belongs to the Sobolev space $W^{1,\infty}(B^r)$ over the unit ball $B^r \subset \mathbb{R}^r$. By Rademacher's theorem, $\phi(z)$ is differentiable almost everywhere, allowing us to invoke the Gauss-Green / Gauss-Ostrogradsky theorem in the weak sense (via weak derivatives). This transforms the boundary flux integral over $S^{r-1}$ into a volume gradient integration over the intrinsic unit ball $B^r \subset \mathbb{R}^r$:
\begin{equation*}
    \int_{S^{r-1}} f(x_t + \mu_t Q_r \Lambda_r^{1/2} w) w \, dS(w) = \int_{B^r} \nabla_{z} \left( f(x_t + \mu_t Q_r \Lambda_r^{1/2} z) \right) \, dz.
\end{equation*}
Applying the multi-variable chain rule under the intrinsic spatial differentiation mapping yields $\nabla_z = \mu_t \Lambda_r^{1/2} Q_r^\top \nabla_x$. Since $\nabla_x f$ acts only within $\mathcal{S}$, this corresponds strictly to the subspace gradient $\nabla_{\mathcal{S}} f_{\mu_t, \hat{\Sigma}}(x_t) \triangleq Q_r^\top \nabla_x f_{\mu_t, \hat{\Sigma}}(x_t) \in \mathbb{R}^r$. The volume integral relates as:
\begin{equation*}
    \int_{B^r} \nabla_{z} f(x_t + \mu_t Q_r \Lambda_r^{1/2} z) \, dz = \mu_t \Lambda_r^{1/2} \int_{B^r} Q_r^\top \nabla_x f(x_t + \mu_t Q_r \Lambda_r^{1/2} z) \, dz.
\end{equation*}
Multiplying by the inverse operator $\frac{1}{\mu_t} \Lambda_r^{-1/2}$ and scaling by the standard volume-to-surface quotient $\frac{\operatorname{area}(S^{r-1})}{\operatorname{vol}(B^r)} = r$, the intrinsic expectation reduces strictly to:
\begin{equation*}
    \mathbb{E}_{w} \left[ f(x_t + \mu_t Q_r \Lambda_r^{1/2} w) w \right] = \frac{\mu_t}{r} \Lambda_r^{1/2} Q_r^\top \nabla f_{\mu_t, \hat{\Sigma}}(x_t).
\end{equation*}
Substituting this back into the ambient embedding in Eq.~\eqref{eq:boundary_flux_w} and multiplying by $Q_r$ maps the identity onto $\mathbb{R}^d$:
\begin{equation}
\label{eq:strict_sphere_identity}
    \mathbb{E}_{v_t} \left[ f(x_t + \mu_t \hat{\Sigma}^{1/2} v_t) v_t \right] = \frac{\mu_t}{r} \left( Q_r \Lambda_r^{1/2} Q_r^\top \right) \nabla f_{\mu_t, \hat{\Sigma}}(x_t) = \frac{\mu_t}{r} \hat{\Sigma}^{1/2} \nabla f_{\mu_t, \hat{\Sigma}}(x_t).
\end{equation}

Finally, substituting the strict, dimensionally aligned geometric identity from Eq.~\eqref{eq:strict_sphere_identity} back into Eq.~\eqref{eq:clean_expectation_forward} yields:
\begin{equation*}
    \mathbb{E}[\tilde{g}_t \mid \mathcal{F}_{t-1}] = \frac{r}{\mu_t} \hat{\Sigma}^{+1/2} \left( \frac{\mu_t}{r} \hat{\Sigma}^{1/2} \nabla f_{\mu_t, \hat{\Sigma}}(x_t) \right) = \left( \hat{\Sigma}^{+1/2} \hat{\Sigma}^{1/2} \right) \nabla f_{\mu_t, \hat{\Sigma}}(x_t).
\end{equation*}
Invoking the spectral pseudoinverse reduction, the product of the fractional matrix roots maps onto the symmetric orthogonal projection matrix: $\hat{\Sigma}^{+1/2} \hat{\Sigma}^{1/2} = Q_r \Lambda_r^{-1/2} \Lambda_r^{1/2} Q_r^\top = Q_r Q_r^\top \triangleq P_{\mathcal{S}}$, leaving:
\begin{equation*}
    \mathbb{E}[\tilde{g}_t \mid \mathcal{F}_{t-1}] = P_{\mathcal{S}} \nabla f_{\mu_t, \hat{\Sigma}}(x_t),
\end{equation*}
which establishes the strict, non-asymptotic unbiasedness of the space-transformed estimator on the active optimization subspace $\mathcal{S}$ under weak Sobolev derivatives.
\end{proof}

\subsection{Proof of Lemma \ref{lem:norm-anisotropic}}
\label{app:norm-anisotropic}

\begin{proof}
Consider the difference of function values at the two sampling points under a common random noise realization $\xi_t$:

\begin{equation*}
    \Delta F_t := F(x_t + \mu_t \hat{\Sigma}^{1/2} v_t; \xi_t) - F(x_t - \mu_t \hat{\Sigma}^{1/2} v_t; \xi_t).
\end{equation*}

Since $F(\cdot;\xi)$ is $L$-Lipschitz continuous almost surely and the noise is identically shared, the stochastic shift cancels out completely, yielding the following deterministic bound:

\begin{align*}
    |\Delta F_t| &\leqslant L \| (x_t + \mu_t \hat{\Sigma}^{1/2} v_t) - (x_t - \mu_t \hat{\Sigma}^{1/2} v_t) \|_2 \\
    &= 2 L \mu_t \|\hat{\Sigma}^{1/2} v_t\|_2.
\end{align*}

Now, we evaluate the Euclidean norm of the space-transformed gradient estimator $\tilde{g_t}$ by combining the functional variation and the preconditioning scale:

\begin{align*}
    \|\tilde{g_t}\|_2 &= \left\| \frac{r}{2\mu_t} \Delta F_t \hat{\Sigma}^{+1/2} v_t \right\|_2 = \frac{r}{2\mu_t} |\Delta F_t| \cdot \|\hat{\Sigma}^{+1/2} v_t\|_2.
\end{align*}

Substituting the deterministic upper bound for $|\Delta F_t|$ directly into the norm formulation yields:

\begin{align*}
    \|\tilde{g_t}\|_2 &\leqslant \frac{r}{2\mu_t} \cdot \left(2 L \mu_t \|\hat{\Sigma}^{1/2} v_t\|_2 \right) \cdot \|\hat{\Sigma}^{+1/2} v_t\|_2 \\
    &= L r \cdot \left( \|\hat{\Sigma}^{1/2} v_t\|_2 \cdot \|\hat{\Sigma}^{+1/2} v_t\|_2 \right).
\end{align*}

We invoke the truncated singular value decomposition within the active subspace coordinates, where $\|\hat{\Sigma}^{1/2} v_t\|_2 = \sqrt{\sum_{i=1}^r \lambda_i v_{t,i}^2}$ and $\|\hat{\Sigma}^{+1/2} v_t\|_2 = \sqrt{\sum_{i=1}^r \lambda_i^{-1} v_{t,i}^2}$ with $\sum_{i=1}^r v_{t,i}^2 = 1$. By applying the square-root variant of Kantorovich's inequality for complementary spectral weights, the product of these dual coordinate norms is strictly bounded by the underlying geometry as:

\begin{equation*}
    \|\hat{\Sigma}^{1/2} v_t\|_2 \cdot \|\hat{\Sigma}^{+1/2} v_t\|_2 \leqslant \frac{\lambda_{\max}(\hat{\Sigma}) + \lambda_r(\hat{\Sigma})}{2 \sqrt{\lambda_{\max}(\hat{\Sigma}) \lambda_r(\hat{\Sigma})}} = \frac{1}{2}\left(\sqrt{\kappa(\hat{\Sigma})} + \frac{1}{\sqrt{\kappa(\hat{\Sigma})}}\right) \leqslant \sqrt{\kappa(\hat{\Sigma})},
\end{equation*}
where $\kappa(\hat{\Sigma}) \triangleq \lambda_{\max}(\hat{\Sigma})/\lambda_r(\hat{\Sigma}) \geqslant 1$ represents the intrinsic condition number of the empirical covariance matrix on $\mathcal{S}$. This directly implies the tight, non-asymptotic deterministic upper bound:
\begin{equation*}
    \|\tilde{g_t}\|_2 \leqslant L r \sqrt{\kappa(\hat{\Sigma})},
\end{equation*}
which completes the proof.
\end{proof}

\subsection{Proof of Lemma \ref{lem:second-anisotropic}}
\label{app:second-anisotropic}

\begin{proof}
Consider the difference of function values at the two sampling points under a common random noise realization $\xi_t$:
\begin{equation*}
    \Delta F_t := F(x_t + \mu_t \hat{\Sigma}^{1/2} v_t; \xi_t) - F(x_t - \mu_t \hat{\Sigma}^{1/2} v_t; \xi_t).
\end{equation*}
Under the $L$-Lipschitz continuity assumption, $F(\cdot; \xi)$ yields the strict deterministic upper bound:
\begin{equation*}
    (\Delta F_t)^2 \leqslant 4 L^2 \mu_t^2 \|\hat{\Sigma}^{1/2} v_t\|_2^2 = 4 L^2 \mu_t^2 v_t^\top \hat{\Sigma} v_t.
\end{equation*}
Now, we evaluate the expected squared norm of the space-transformed gradient estimator $\tilde{g_t}$ defined in Eq.~\ref{eq:subspace_estimator}. Using the algebraic identity $\|x\|_2^2 = \operatorname{tr}(xx^\top)$ and pulling out scalar terms, we have:
\begin{equation*}
    \|\tilde{g_t}\|_2^2 = \left\| \frac{r}{2\mu_t} \Delta F_t \hat{\Sigma}^{+1/2} v_t \right\|_2^2 = \frac{r^2 (\Delta F_t)^2}{4\mu_t^2} \operatorname{tr}\left( \hat{\Sigma}^{+1/2} v_t v_t^\top \hat{\Sigma}^{+1/2} \right).
\end{equation*}
Applying the cyclic property of the trace and substituting the deterministic bound for $(\Delta F_t)^2$ directly into the formulation, we obtain:
\begin{align*}
    \|\tilde{g_t}\|_2^2 &\leqslant \frac{r^2 \left(4 L^2 \mu_t^2 v_t^\top \hat{\Sigma} v_t \right)}{4\mu_t^2} \operatorname{tr}\left( \hat{\Sigma}^{+} v_t v_t^\top \right) \\
    &= L^2 r^2 \left( v_t^\top \hat{\Sigma} v_t \right) \cdot \left( v_t^\top \hat{\Sigma}^{+} v_t \right),
\end{align*}
where we used the identity $\operatorname{tr}(\hat{\Sigma}^{+} v_t v_t^\top) = v_t^\top \hat{\Sigma}^{+} v_t$, and $\hat{\Sigma}^{+}$ denotes the Moore-Penrose pseudoinverse.

To bound the product of these dual quadratic forms under the uniform distribution on the sphere, we invoke Kantorovich's inequality for complementary spectral weights. For any unit vector $v_t \in \mathbb{S}^{r-1}$ within the active subspace $\mathcal{S}$, the product is strictly bounded by the underlying geometry as:
\begin{equation*}
    \left( v_t^\top \hat{\Sigma} v_t \right) \cdot \left( v_t^\top \hat{\Sigma}^{+} v_t \right) \leqslant \frac{(\lambda_{\max}(\hat{\Sigma}) + \lambda_r(\hat{\Sigma}))^2}{4 \lambda_{\max}(\hat{\Sigma}) \lambda_r(\hat{\Sigma})} \leqslant \kappa(\hat{\Sigma}),
\end{equation*}
where $\kappa(\hat{\Sigma}) \triangleq \lambda_{\max}(\hat{\Sigma})/\lambda_r(\hat{\Sigma}) \geqslant 1$ represents the condition number inside $\mathcal{S}$. This implies a tight, non-asymptotic deterministic pointwise bound on the estimator's norm:
\begin{equation*}
    \|\tilde{g_t}\|_2^2 \leqslant L^2 r^2 \kappa(\hat{\Sigma}).
\end{equation*}

By taking the conditional expectation with respect to $\mathcal{F}_{t-1}$ on both sides, the relation holds cleanly. This tight bound demonstrates that the second-moment variance of the space-transformed estimator scales strictly with the squared active rank $r^2$ and the active subspace condition number $\kappa(\hat{\Sigma})$, completely avoiding any dependency on the ambient dimension $d$. This completes the proof.

\end{proof}

\subsection{Proof of Lemma \ref{lem:weighted-regret}}
\label{app:weighted-regret}

\begin{proof}
Let $s_k \triangleq \|x_k - x_*\|_2$. The standard Euclidean projection property onto the compact convex domain $\mathcal{X}$ under the projected SGD update rule $x_{k+1} = \Pi_{\mathcal{X}}(x_k - \eta_k \tilde{g}_k)$ implies the standard distance contraction inequality:
\begin{equation*}
    \|x_{k+1} - x_*\|_2^2 \leqslant \|x_k - x_* - \eta_k \tilde{g}_k\|_2^2 = \|x_k - x_*\|_2^2 - 2\eta_k \langle \tilde{g}_k, x_k - x_* \rangle + \eta_k^2 \|\tilde{g}_k\|_2^2.
\end{equation*}
Rearranging the terms and isolating the scalar product product yields:
\begin{equation*}
    2\langle \tilde{g}_k, x_k - x_* \rangle \leqslant \frac{s_k^2 - s_{k+1}^2}{\eta_k} + \eta_k \|\tilde{g}_k\|_2^2.
\end{equation*}
Multiplying both sides by the non-negative dynamic mobility weight $\bar{r}_k$ and substituting the parameter schedule definition $\eta_k = \bar{r}_k / \sqrt{G_k}$, the weight cancels out identically from the distance denominator, yielding the pointwise relation:
\begin{equation}
\label{eq:abel_pointwise_strict_start}
    2\bar{r}_k \langle \tilde{g}_k, x_k - x_* \rangle \leqslant \sqrt{G_k}(s_k^2 - s_{k+1}^2) + \frac{\bar{r}_k^2 \|\tilde{g}_k\|_2^2}{\sqrt{G_k}}.
\end{equation}
We sum Eq.~\eqref{eq:abel_pointwise_strict_start} from $k=0$ to $t-1$. We denote the cumulative sums on the right-hand side as $A$ and $B$ respectively, such that $2 R_t \leqslant A + B$, where $R_t \triangleq \sum_{k=0}^{t-1} \bar{r}_k \langle \tilde{g}_k, x_k - x_* \rangle$.

To evaluate the first summation $A \triangleq \sum_{k=0}^{t-1} \sqrt{G_k}(s_k^2 - s_{k+1}^2)$, we perform an exact Abel transformation (summation by parts). Crucially, following the canonical baseline architecture of Lemma~1 in Ivgi et al.~\cite{ivgi2023dog}, we strictly preserve the negative boundary term at the horizon edge:
\begin{align*}
    A &= \sqrt{G_0}s_0^2 - \sqrt{G_0}s_1^2 + \sqrt{G_1}s_1^2 - \sqrt{G_1}s_2^2 + \dots + \sqrt{G_{t-1}}s_{t-1}^2 - \sqrt{G_{t-1}}s_t^2 \\
    &= \sqrt{G_0}s_0^2 - \sqrt{G_{t-1}}s_t^2 + \sum_{k=1}^{t-1} (\sqrt{G_k} - \sqrt{G_{k-1}})s_k^2.
\end{align*}
Let $j \in \arg\max_{k \leqslant t} s_k$ define the index maximizing the distance sequence up to horizon $t$, such that $s_j = \bar{s}_t$. Since the radical gradient energy sequence $\sqrt{G_k}$ is monotonically non-decreasing, the sum of differences telescopes exactly as $\sqrt{G_0} + \sum_{k=1}^{t-1} (\sqrt{G_k} - \sqrt{G_{k-1}}) = \sqrt{G_{t-1}}$. Upper-bounding the current states $s_k \leqslant s_j$ inside the summation yields:
\begin{align*}
    A &\leqslant \sqrt{G_0}s_j^2 - \sqrt{G_{t-1}}s_t^2 + \sum_{k=1}^{t-1} (\sqrt{G_k} - \sqrt{G_{k-1}})s_j^2 \\
    &= s_j^2 \left( \sqrt{G_0} + \sum_{k=1}^{t-1} (\sqrt{G_k} - \sqrt{G_{k-1}}) \right) - \sqrt{G_{t-1}}s_t^2 \\
    &= \sqrt{G_{t-1}} s_j^2 - \sqrt{G_{t-1}}s_t^2 = \sqrt{G_{t-1}}(\bar{s}_t^2 - s_t^2).
\end{align*}
To control the difference of the squared running maximum and the boundary state without introducing ungrounded dimensions, we factor the algebraic identity via the triangle inequality and the non-decreasing spatial mobility constraints:
\begin{align*}
    \bar{s}_t^2 - s_t^2 &= (s_j - s_t)(s_j + s_t) \\
    &\leqslant \|x_j - x_t\|_2 (s_j + s_t) \\
    &\leqslant \left( \|x_j - x_0\|_2 + \|x_t - x_0\|_2 \right) (\bar{s}_t + \bar{s}_t) \\
    &\leqslant (\bar{r}_t + \bar{r}_t)(2\bar{s}_t) = 4 \bar{r}_t \bar{s}_t.
\end{align*}
Substituting this tight geometric majorization back into the Abel outcome isolates the strict bound for the first process:
\begin{equation}
\label{eq:abel_term_A_final_done}
    A \leqslant 4 \bar{r}_t \bar{s}_t \sqrt{G_{t-1}}.
\end{equation}

Next, we evaluate the second summation $B \triangleq \sum_{k=0}^{t-1} \frac{\bar{r}_k^2 \|\tilde{g}_k\|_2^2}{\sqrt{G_k}}$. Since the mobility sequence $\{\bar{r}_k\}$ is positive and strictly non-decreasing, we legally majorize the numerator via its maximum target endpoint $\bar{r}_k \leqslant \bar{r}_{t-1} \leqslant \bar{r}_t$. Invoking the standard continuous radical squeeze identity $\sum_{k=0}^{t-1} \frac{\|\tilde{g}_k\|_2^2}{\sqrt{G_k}} \leqslant 2\sqrt{G_{t-1}}$, we have:
\begin{equation}
\label{eq:abel_term_B_final_done}
    B \leqslant \bar{r}_t^2 \sum_{k=0}^{t-1} \frac{\|\tilde{g}_k\|_2^2}{\sqrt{G_k}} \leqslant 2 \bar{r}_t^2 \sqrt{G_{t-1}}.
\end{equation}

Finally, substituting the corrected boundaries from Eq.~\eqref{eq:abel_term_A_final_done} and Eq.~\eqref{eq:abel_term_B_final_done} back into the master relation $2 R_t \leqslant A + B$ collapses the system cleanly:
\begin{align*}
    2 \sum_{k=0}^{t-1} \bar{r}_k \langle \tilde{g}_k, x_k - x_* \rangle &\leqslant 4 \bar{r}_t \bar{s}_t \sqrt{G_{t-1}} + 2 \bar{r}_t^2 \sqrt{G_{t-1}} \\
    &= 2\bar{r}_t (2\bar{s}_t + \bar{r}_t) \sqrt{G_{t-1}}.
\end{align*}
Dividing both sides by 2 recovers exactly the uniform high-probability statement of Lemma~8:
\begin{equation*}
    \sum_{k=0}^{t-1} \bar{r}_k \langle \tilde{g}_k, x_k - x_* \rangle \leqslant \bar{r}_t (2\bar{s}_t + \bar{r}_t) \sqrt{G_{t-1}},
\end{equation*}
which completes the proof.
\end{proof}

\subsection{Proof of Lemma \ref{lem:martingale-noise}}
\label{app:martingale-noise}
\begin{proof}
Define the filtration $\mathcal{F}_k = \sigma(v_i, \xi_i, 0 \leqslant i \leqslant k)$ for $k \in \mathbb{N}$ and $\mathcal{F}_{-1} = \{\emptyset, \Omega\}$. Following the structural framework of Ivgi et al.~\cite{ivgi2023dog}, we decouple the monotonic distance scaling from the stochastic noise process by defining the normalized martingale difference sequence $(X_k, k \in \mathbb{N})$ and the predictable sequence $(\hat{X}_k, k \in \mathbb{N})$ as follows:
\begin{equation*}
    X_k \triangleq \frac{1}{\bar{s}_k} \langle \Delta_k, x_k - x_* \rangle \in \mathcal{F}_k, \quad \text{and} \quad \hat{X}_k \triangleq -\frac{1}{\bar{s}_k} \langle \nabla f_{\mu_k, \hat{\Sigma}}(x_k), x_k - x_* \rangle \in \mathcal{F}_{k-1}.
\end{equation*}
Since $\bar{s}_k \in \mathcal{F}_{k-1}$, taking the conditional expectation and invoking Lemma~\ref{lem:unbiased} yields:
\begin{equation*}
    \mathbb{E}[X_k \mid \mathcal{F}_{k-1}] = \frac{1}{\bar{s}_k} \langle \mathbb{E}[\tilde{g}_k \mid \mathcal{F}_{k-1}] - P_{\mathcal{S}}\nabla f_{\mu_k, \hat{\Sigma}}(x_k), x_k - x_* \rangle = 0,
\end{equation*}
which confirms that $(X_k, \mathcal{F}_k)$ is a strict martingale difference process. 

Next, we establish a uniform static bound on the increments. By Lemma~6, the norm of the space-transformed anisotropic gradient estimator is bounded almost surely as $\|\tilde{g}_k\|_2 \leqslant L r \sqrt{\kappa(\hat{\Sigma})}$. By the law of total expectation, the expected gradient satisfies $\|P_S\nabla f_{\mu_k, \hat{\Sigma}}(x_k)\|_2 = \|\mathbb{E}[\tilde{g}_k \mid \mathcal{F}_{k-1}]\|_2 \leqslant L r \sqrt{\kappa(\hat{\Sigma})}$, which directly implies:
\begin{equation*}
    \|\Delta_k\|_2 \leqslant \|\tilde{g}_k\|_2 + \|P_S\nabla f_{\mu_k, \hat{\Sigma}}(x_k)\|_2 \leqslant 2 L r \sqrt{\kappa(\hat{\Sigma})}.
\end{equation*}
Applying the Cauchy-Schwarz inequality and leveraging the fact that $\frac{\|x_k - x_*\|_2}{\bar{s}_k} \leqslant 1$ by definition, the absolute magnitude of both processes is bounded almost surely by a static, dimension-invariant uniform constant:
\begin{equation*}
    |X_k| \leqslant \|\Delta_k\|_2 \frac{\|x_k - x_*\|_2}{\bar{s}_k} \leqslant 2 L r \sqrt{\kappa(\hat{\Sigma})} \triangleq c, \quad \text{and} \quad |\hat{X}_k| \leqslant c.
\end{equation*}

Now, define the sequence of product weights as $Y_k \triangleq \bar{r}_k \bar{s}_k$. Because both $\{\bar{r}_k\}$ and $\{\bar{s}_k\}$ are non-negative and non-decreasing sequences by construction, the sequence $Y_k$ belongs to the set of valid non-decreasing weights $\mathcal{S}$. We invoke Freedman's style time-uniform martingale concentration inequality for sub-exponential difference increments scaled by non-decreasing weights (Lemma~12 / Lemma~D.2 in Ivgi et al.~\cite{ivgi2023dog}). For all $\delta \in (0, 1)$, it holds that:
\begin{align*}
    \Pr \left( \exists t \leqslant T : \left| \sum_{k=0}^{t-1} \bar{r}_k \langle \Delta_k, x_k - x_* \rangle \right| \geqslant b_t \right) &= \\ \Pr \left( \exists t \leqslant T : \left| \sum_{k=0}^{t-1} Y_k X_k \right| \geqslant b_t \right) &\leqslant \\ \Pr \left( \exists t \leqslant T, \exists \{y_k\} \in \mathcal{S} : \left| \sum_{k=0}^{t-1} y_k X_k \right| \geqslant b_t \right) &\leqslant \delta,
\end{align*}
where the uniform boundary process $b_t$ scales with the accumulated step variance exactly as specified:
\begin{equation}
    b_t = 8 Y_{t-1} \sqrt{\theta_{t,\delta} \sum_{k=0}^{t-1} (X_k - \hat{X}_k)^2 + c^2 \theta_{t,\delta}^2}.
\end{equation}
Thanks to the negative sign in the definition of $\hat{X}_k$, the difference inside the summation linearizes into a clean, uncoupled projection term:
\begin{align*}
    \sum_{k=0}^{t-1} (X_k - \hat{X}_k)^2 &= \\ \sum_{k=0}^{t-1} \left( \frac{1}{\bar{s}_k} \langle \tilde{g}_k, x_k - x_* \rangle \right)^2 \leqslant \\ \sum_{k=0}^{t-1} \frac{\|x_k - x_*\|_2^2}{\bar{s}_k^2} \|\tilde{g}_k\|_2^2 \leqslant \sum_{k=0}^{t-1} \|\tilde{g}_k\|_2^2 = G_{t-1}.
\end{align*}
Recalling that $Y_{t-1} = \bar{r}_{t-1} \bar{s}_{t-1}$ and substituting the cumulative gradient energy upper bound directly into the radical, we obtain the exact uniform high-probability bound:
\begin{equation}
    b_t = 8 \bar{r}_{t-1} \bar{s}_{t-1} \sqrt{\theta_{t,\delta} G_{t-1} + 4 L^2 r^2 \kappa(\hat{\Sigma}) \theta_{t,\delta}^2},
\end{equation}
where we substituted the static increment parameter $c = 2 L r \sqrt{\kappa(\hat{\Sigma})}$, which holds uniformly for all horizons $t \geqslant 1$, completing the proof.
\end{proof}

\subsection{Proof of Lemma \ref{lem:smoothing-noise}}
\label{app:smoothing-noise}

\begin{proof}
Recall that under the ellipsoidal smoothing definition with $\lambda_{\max}(\hat{\Sigma}) = 1$, the absolute bias satisfies $|f_{\mu_k, \hat{\Sigma}}(x_k) - f(x_k)| \leqslant L \mu_k$. Substituting the adaptive trajectory $\mu_k = \bar{r}_k \sqrt{\frac{d^*}{k+1}}$ into the cumulative error summation yields:
\begin{align*}
    \sum_{k=0}^{t-1} 2 L \bar{r}_k \mu_k &\leqslant 2 L \bar{r}_{t-1}^2 \sqrt{d^*} \sum_{k=0}^{t-1} \frac{1}{\sqrt{k+1}} \leqslant 4 L \bar{r}_{t-1}^2 \sqrt{d^* t},
\end{align*}
where we applied the standard continuous integral bound $\sum_{k=0}^{t-1} \frac{1}{\sqrt{k+1}} \leqslant 2\sqrt{t}$. This completes the proof.
\end{proof}

\subsection{Proof of Lemma \ref{lem:weights-lower}}
\label{app:weights-lower}

\begin{proof}
By invoking Lemma~3 of Ivgi et al.~\cite{ivgi2023dog}, for any positive non-decreasing sequence $\{a_t\}_{t=0}^T$, the cumulative ratio is bounded uniformly by:
\begin{equation}
\label{eq:dog_original_lower_bound}
    M_T \triangleq \max_{t \in [T]} \frac{\sum_{k=0}^{t-1} a_k}{a_t} \geqslant \frac{T}{e \log_+\left(\frac{a_T}{a_0}\right)} - 1,
\end{equation}
where $\log_+ u \triangleq 1 + \log u$. Substituting the algorithmic mobility weights $a_t \triangleq \bar{r}_t$ and the baseline constant $a_0 \triangleq r_\epsilon$, let $z \triangleq \frac{T}{e \log_+\left(\bar{r}_T / r_\epsilon\right)}$ denote the continuous scaling progression factor. Eq.~\eqref{eq:dog_original_lower_bound} thus establishes $M_T \geqslant z - 1$.

To close the mathematical gap caused by the negative offset $-1$ for small horizons, we exploit the intrinsic geometric structure of Algorithm~1. On the very first active iteration $t=1$, the traveled distance satisfies $\|x_1 - x_0\|_2 \leqslant \bar{r}_0$ by the contractive projection mapping. By Line~6 of Algorithm~1, the update rule enforces $\bar{r}_1 = \max\{\bar{r}_0, \|x_1 - x_0\|_2\} = \bar{r}_0 = r_\epsilon$. Consequently, the initial ratio is strictly bounded bounded below by a hard deterministic threshold:
\begin{equation*}
    M_T \geqslant \frac{\bar{r}_0}{\bar{r}_1} = \frac{r_\epsilon}{r_\epsilon} = 1.
\end{equation*}
Combining the independent boundaries via a uniform maximum envelope, the cumulative weight ratio satisfies $M_T \geqslant \max\{1, z - 1\}$ for all valid operating regimes. We now evaluate this system under two complementary cases:
\begin{enumerate}
    \item \textbf{Case 1 ($z \leqslant 2$):} If the sequence progression is slow such that $z \leqslant 2$, the inequality is trivially bounded by the lower threshold: $\max\{1, z - 1\} = 1 \geqslant \frac{z}{2}$.
    \item \textbf{Case 2 ($z > 2$):} If the horizon is large such that $z > 2$, the algebraic subtraction satisfies $z - 1 > \frac{z}{2}$, which directly implies $\max\{1, z - 1\} = z - 1 \geqslant \frac{z}{2}$.
\end{enumerate}
Therefore, the uniform inequality $\max\{1, z - 1\} \geqslant \frac{z}{2}$ holds holds unconditionally for all $T \geqslant 1$. Substituting the full expression for $z$ into the consolidated lower boundary yields:
\begin{equation*}
    \max_{t \in [T]} \sum_{k=0}^{t-1} \frac{\bar{r}_k}{\bar{r}_t} \geqslant \frac{z}{2} = \frac{T}{2e \log_+ \left( \frac{\bar{r}_T}{r_\epsilon} \right)},
\end{equation*}
which rigorously establishes the stated contractive bound. The proof is complete.
\end{proof}

\section{Convergence of POEM-ES}

\subsection{Proof of Theorem \ref{thm:convergence-POEM-ES}}
\label{app:convergence-POEM-ES}

\begin{proof}
Let $t \leqslant T$ be an arbitrary intermediate iteration horizon. By the convexity of the objective function $f(x)$ and applying Jensen's inequality over the step trajectory up to iteration $t-1$, the functional gap at the ergodic averaged iterate $\bar{x}_t \triangleq \frac{\sum_{k=0}^{t-1} \bar{r}_k x_k}{\sum_{k=0}^{t-1} \bar{r}_k}$ is upper-bounded by:
\begin{equation*}
    f(\bar{x}_t) - f(x_*) \leqslant \frac{1}{\sum_{k=0}^{t-1} \bar{r}_k} \sum_{k=0}^{t-1} \bar{r}_k \bigl( f(x_k) - f(x_*) \bigr).
\end{equation*}
To incorporate the ellipsoidal randomized smoothing layer, we decompose the true objective gap at each step $k$ through the smoothed surrogate $f_{\mu_k, \hat{\Sigma}}(x_k)$. Under the global Lipschitz property established in Assumption~\ref{asm:convex_lipschitz} and applying the smoothing bias bound from Lemma~\ref{lem:smoothing-noise} under $\lambda_{\max}(\hat{\Sigma}) = 1$, the function evaluations at the perturbed oracle points $x_k \pm \mu_k \hat{\Sigma}^{1/2}v_k \in \mathbb{R}^d$ are strictly valid, yielding:
\begin{align*}
    f(x_k) - f(x_*) &= \\ \bigl( f(x_k) - f_{\mu_k, \hat{\Sigma}}(x_k) \bigr) + \bigl( f_{\mu_k, \hat{\Sigma}}(x_k) - f_{\mu_k, \hat{\Sigma}}(x_*) \bigr) + \bigl( f_{\mu_k, \hat{\Sigma}}(x_*) - f(x_*) \bigr)
    &\leqslant \\ \langle \nabla f_{\mu_k, \hat{\Sigma}}(x_k), x_k - x_* \rangle + 2 L \mu_k.
\end{align*}

We now evaluate the subspace alignment of the displacement vector. By the projection invariance in Assumption~\ref{ass:subspace_invariance}, since $x_0 \in x_0 + \mathcal{S}$ and $x_{k} = \Pi_{\mathcal{X}}(x_{k-1} - \eta_{k-1}\tilde{g}_{k-1})$, mathematical induction guarantees that $x_k \in x_0 + \mathcal{S}$, which implies $(x_k - x_0) \in \mathcal{S}$ for all $k \geqslant 0$. Furthermore, by the optimum compatibility condition in Assumption~\ref{ass:subspace_invariance}, the target optimal point satisfies $x_* \in x_0 + \mathcal{S}$, implying $(x_* - x_0) \in \mathcal{S}$. Subtracting these two relations, the displacement vector satisfies:
\begin{equation*}
    x_k - x_* = (x_k - x_0) - (x_* - x_0) \in \mathcal{S}, \quad \forall k \geqslant 0.
\end{equation*}
Since the displacement resides strictly within the active linear subspace, the symmetric orthogonal projector $P_{\mathcal{S}}$ acts as a strict identity operator on the inner product: $\langle \nabla f_{\mu_k, \hat{\Sigma}}(x_k), x_k - x_* \rangle = \langle P_{\mathcal{S}} \nabla f_{\mu_k, \hat{\Sigma}}(x_k), x_k - x_* \rangle$. 

Invoking our strict subspace directionally unbiased expectation rule from Lemma~\ref{lem:unbiased}, we substitute the gradient term with the conditional expectation of the estimator $\mathbb{E}[\tilde{g}_k \mid \mathcal{F}_{k-1}] = P_{\mathcal{S}} \nabla f_{\mu_k, \hat{\Sigma}}(x_k)$, which bypasses any scale discrepancies:
\begin{equation*}
    \langle \nabla f_{\mu_k, \hat{\Sigma}}(x_k), x_k - x_* \rangle = \langle \mathbb{E}[\tilde{g}_k \mid \mathcal{F}_{k-1}], x_k - x_* \rangle = \langle \tilde{g}_k - \Delta_k, x_k - x_* \rangle,
\end{equation*}
where $\Delta_k \triangleq \tilde{g}_k - \mathbb{E}[\tilde{g}_k \mid \mathcal{F}_{k-1}]$. Multiplying by the weight sequence $\{\bar{r}_k\}$ and summing from $k=0$ to $t-1$ yields the exact algebraic regret decomposition for the arbitrary intermediate horizon $t$:
\begin{align}
\label{eq:unified_regret_t}
    \sum_{k=0}^{t-1} \bar{r}_k \bigl( f(x_k) - f(x_*) \bigr) \leqslant  \sum_{k=0}^{t-1} \bar{r}_k \langle \tilde{g}_k, x_k - x_* \rangle - \sum_{k=0}^{t-1} \bar{r}_k \langle \Delta_k, x_k - x_* \rangle + \\ \sum_{k=0}^{t-1} 2 L \bar{r}_k \mu_k.
\end{align}

We bound each process in the right-hand side of Eq.~\eqref{eq:unified_regret_t} for the horizon $t$:
\begin{enumerate}
    \item \textbf{Weighted Regret Control:} By Lemma 8, the parameter-free dynamic step schedule yields $\sum_{k=0}^{t-1} \bar{r}_k \langle \tilde{g}_k, x_k - x_* \rangle \leqslant \bar{r}_{t}(2\bar{s}_{t} + \bar{r}_{t})\sqrt{G_{t-1}}$. By the global geometric compactness of $\mathcal{X}$, the spatial parameters are uniformly majorized by the finite domain diameter as $\bar{s}_{t} \leqslant D_X$ and $\bar{r}_{t} \leqslant D_X$. This yields the valid uniform bound:
    \begin{equation*}
        \sum_{k=0}^{t-1} \bar{r}_k \langle \tilde{g}_k, x_k - x_* \rangle \leqslant 3 D_X \bar{r}_{t} \sqrt{G_{t-1}}.
    \end{equation*}
    
    \item \textbf{Martingale Noise Concentration:} By Lemma~\ref{lem:martingale-noise}, utilizing the predictable coordinate sequence injection $\hat{X}_k \triangleq -\frac{1}{\bar{s}_k} \langle \nabla f_{\mu_k, \hat{\Sigma}}(x_k), x_k - x_* \rangle$, the uniform martingale noise sum is bounded with probability at least $1-\delta$ by:
    \begin{equation*}
        - \sum_{k=0}^{t-1} \bar{r}_k \langle \Delta_k, x_k - x_* \rangle \leqslant b_t = 8 Y_{t-1} \sqrt{\theta_{t,\delta} \sum_{k=0}^{t-1} (X_k - \hat{X}_k)^2 + c^2 \theta_{t,\delta}^2}.
    \end{equation*}
    Substituting the algebraic variance proxy $\sum_{k=0}^{t-1} (X_k - \hat{X}_k)^2 \leqslant G_{t-1}$, the weight constraint $Y_{t-1} = \bar{r}_{t-1}\bar{s}_{t-1} \leqslant D_X \bar{r}_{t-1}$, and the sub-exponential limit $c = 2 L r \sqrt{\kappa(\hat{\Sigma})}$, the radical matches exactly as:
    \begin{equation*}
        - \sum_{k=0}^{t-1} \bar{r}_k \langle \Delta_k, x_k - x_* \rangle \leqslant 8 D_X \bar{r}_{t-1} \sqrt{\theta_{t,\delta} G_{t-1} + 4 L^2 r^2 \kappa(\hat{\Sigma}) \theta_{t,\delta}^2}.
    \end{equation*}
    
    \item \textbf{Ellipsoidal Smoothing Bias Control:} By Lemma~\ref{lem:smoothing-noise}, substituting the active smoothing trajectory $\mu_k = \bar{r}_k \sqrt{d^*/(k+1)}$ under $\lambda_{\max}(\hat{\Sigma}) = 1$, the cumulative spatial error satisfies:
    \begin{equation*}
        \sum_{k=0}^{t-1} 2 L \bar{r}_k \mu_k \leqslant 4 L \bar{r}_{t-1}^2 \sqrt{d^* t}.
    \end{equation*}
\end{enumerate}

Combining these three bounds into Eq.~\eqref{eq:unified_regret_t}, applying the sub-additive radical property $\sqrt{a+b} \leqslant \sqrt{a} + \sqrt{b}$ to the concentrated noise term, and substituting the pointwise total energy bound $G_{t-1} \leqslant t \cdot L^2 r^2 \kappa(\hat{\Sigma})$, the running weighted sum relates for the arbitrary intermediate horizon $t$ as:
\begin{equation*}
    \sum_{k=0}^{t-1} \bar{r}_k \bigl( f(x_k) - f(x_*) \bigr) \leqslant 11 D_X \bar{r}_{t} L r \sqrt{\kappa(\hat{\Sigma}) \theta_{t,\delta} t} + 16 D_X \bar{r}_{t} L r \sqrt{\kappa(\hat{\Sigma})} \theta_{t,\delta} + 4 L \bar{r}_{t}^2 \sqrt{d^* t}.
\end{equation*}
Dividing both sides by the weight summation $\sum_{k=0}^{t-1} \bar{r}_k$ yields a strict uniform functional gap bound for any arbitrary step $t \leqslant T$:
\begin{align}
\label{eq:final_arbitrary_t_ineq}
    f(\bar{x}_t) - f(x_*) \leqslant \frac{\bar{r}_{t-1}}{\sum_{k=0}^{t-1} \bar{r}_k} \left( 11 D_X L r \sqrt{\kappa(\hat{\Sigma}) \theta_{t,\delta} t} + 4 L \bar{r}_{t} \sqrt{d^* t} \right) + \\ \frac{\bar{r}_{t}}{\sum_{k=0}^{t-1} \bar{r}_k} \left( 16 D_X L r \sqrt{\kappa(\hat{\Sigma})} \theta_{t,\delta} \right).
\end{align}

We now specialize the arbitrary path inequality in Eq.~\eqref{eq:final_arbitrary_t_ineq} by evaluating it strictly at the output step $t = \tau_T$ returned by Algorithm~1, where $\tau_T \triangleq \arg\max_{t \leqslant T} \sum_{k=0}^{t-1} \frac{\bar{r}_k}{\bar{r}_t}$. This specialization immediately implies $\frac{\bar{r}_{\tau_T-1}}{\sum_{k=0}^{\tau_T-1} \bar{r}_k} \leqslant \frac{\bar{r}_{\tau_T}}{\sum_{k=0}^{\tau_T-1} \bar{r}_k}$. Invoking Lemma~\ref{lem:weights-lower}, the reciprocal weight ratio contracts sharply as $\frac{\bar{r}_{\tau_T}}{\sum_{k=0}^{\tau_T-1} \bar{r}_k} \leqslant \frac{2e \log_+(D_X / r_\epsilon)}{T}$.

Substituting this contractive bound into Eq.~\eqref{eq:final_arbitrary_t_ineq} and noting that $\sqrt{\tau_T} \leqslant \sqrt{T}$ holds trivially on the compact search budget, the components resolve as:
\begin{align*}
    f(\bar{x}_{\tau_T}) - f(x_*) \leqslant \\ \frac{2e \log_+(D_X / r_\epsilon)}{T} \left( 11 D_X L r \sqrt{\kappa(\hat{\Sigma}) \theta_{T,\delta} T} + 4 L D_X \sqrt{d^* T} \right) + \mathcal{O}\left(\frac{r\sqrt{\kappa(\hat{\Sigma})}}{T}\right).
\end{align*}
Simplifying the terms and suppressing the higher-order martingale remainder $\mathcal{O}(T^{-1})$ as it decays asymptotically faster than the leading sublinear convergence pacing $\mathcal{O}(T^{-1/2})$, we recover the clean high-probability bound for the true algorithm output $\bar{x}_{\tau_T}$:
\begin{equation*}
    f(\bar{x}_{\tau_T}) - f(x_*) \leqslant \mathcal{O} \left( \frac{r \sqrt{\kappa(\hat{\Sigma})} + \sqrt{d^*}}{\sqrt{T}} \cdot \theta_{T,\delta} L D_X \log_+\left(\frac{D_X}{r_\epsilon}\right) \right).
\end{equation*}
The proof is complete.
\end{proof}

\end{document}